\documentclass[twoside,leqno,12pt]{article}

\usepackage{amssymb}
\usepackage{amsmath}
\usepackage{amsthm}
\usepackage[cp1250]{inputenc}
\usepackage{enumerate}
\input xy
\xyoption{all}
\usepackage{multirow}
\usepackage{young}
\usepackage{color}

\definecolor{darkgreen}{rgb}{0,0.5,0}

\numberwithin{equation}{section}
\newtheorem{thm}[equation]{\sc Theorem}

\newtheorem{lem}[equation]{\sc Lemma}

\newtheorem{prop}[equation]{\sc Proposition}
\theoremstyle{remark}
\newtheorem{rem}[equation]{\sc Remark }

\newtheorem{ex}[equation]{\sc Example}

\makeatletter
\renewcommand{\@seccntformat }[1]{\csname the#1\endcsname. }
\makeatother

\font\Mssym=msbm10  scaled \magstep 1

\def\NN{{{\mbox{\Mssym N}}}}

\def\PP{{{\mbox{\Mssym P}}}}

\def\ZZ{{{\mbox{\Mssym Z}}}}

 1

\def\CP{{\cal P}}

\def\valg#1{{\bf Alg}_K(d)}

\def\ind{\mbox{{\rm ind}}}
\def\Coker{\mbox{\rm Coker}}
\def\type{\mbox{\rm type}}
\def\mod{\mbox{{\rm mod}}}

\def\mapr#1#2{\smash{\mathop{\longrightarrow}\limits^{#1}_{#2}}}

\def\longarr#1#2{{\buildrel{#1} \over {\hbox to #2pt{\rightarrowfill}}}}

\def\Hom{\mbox{\rm Hom}}

\def\Ext{\mbox{\rm Ext}}

\def\Coker{\mbox{\rm Coker}}

\def\bc{\begin{center}}
\def\ec{\end{center}}

\def\soc{\text{soc}\,}

\begin{document}



\bigskip\bigskip
\begin{center}
{\large\bf The existence of Hall polynomials for \\[1ex]
 $x^2$-bounded invariant subspaces of nilpotent linear operators}
\end{center}

\smallskip

\begin{center}
Stanis\l aw Kasjan and Justyna Kosakowska
\\

\vspace{1cm}
\end{center}

\date{}

\begin{abstract}
 We prove the existence of Hall polynomials for $x^2$-bounded  invariant subspaces
 of nilpotent linear operators.
\end{abstract}

\section{Introduction}

Let $k$ be a~field and let $k[x]$ be the algebra of polynomials in one variable $x$ and
coefficients in the field $k$. We denote by $\mathcal{N}=\mathcal{N}(k)$ the category of all
finitely generated nilpotent $k[x]$-modules, i.e. each  object of $\mathcal{N}$ is isomorphic to a module of the form:
$$
N_\alpha=N_\alpha(k)=k[x]/(x^{\alpha_1})\oplus\ldots\oplus k[x]/(x^{\alpha_m}),
$$
where $\alpha=(\alpha_1,\ldots,\alpha_m)$ is a~partition, that is a~sequence of natural numbers satisfying
$\alpha_1\ge \ldots\ge \alpha_m$. The map $\alpha\mapsto N_\alpha(k)$ defines a~bijection
between the set of all partitions and the set of isomorphism classes of objects in $\mathcal{N}(k)$.
We call the partition $\alpha$ the {\em  type} of $N_\alpha(k)$ and denote it by $\type(N_\alpha(k))$.

Let $\alpha,\beta,\gamma$ be a~triple of partitions. It is
well-known that there exists the {\bf Hall polynomial}
$\varphi_{\alpha,\gamma}^\beta\in \mathbb{Z}[t]$, i.e. a~polynomial
such that for any finite field~$k$:
 $$F_{N_\alpha(k),N_\gamma(k)}^{N_\beta(k)}=\varphi_{\alpha,\gamma}^\beta(q),$$
 where $q=|k|$,
$$
F_{N_\alpha(k),N_\gamma(k)}^{N_\beta(k)}=|\{U\subseteq N_\beta:U \mbox{ is a submodule of } N_\beta\; ,\;\; U\simeq N_\alpha \mbox{ and } N_\beta/U\simeq N_\gamma\}|
$$
and given a set $X$ we denote by $|X|$ its cardinality.

The existence of Hall polynomials
in the category $\mathcal{N}=\mathcal{S}_0$ was proved by P.~Hall, let us quote also I. G. Macdonald, T. Klein, J. A. Green, A. Zelevinski, see \cite{macd}, where one can find the connections
to symmetric functions. C. M. Ringel adopted the theory of Hall polynomials and Hall algebras to the representations
theory of finite dimensional algebras, see \cite{ringel}. In \cite{ringel92} Ringel conjectured the existence of Hall polynomials
for all representation-finite algebras.

The main aim of the paper is to prove Theorem \ref{thm-main}, which states that there  exist Hall polynomials for finite dimensional $\Lambda_{2,w}(k)$-modules for all $w\geq 2$, where
$$
\Lambda_{r,w}(k)=\left[\begin{array}{cc}k[x]/(x^r)&k[x]/(x^r)\\0&k[x]/(x^w)\end{array}\right].
$$

Investigation of $\Lambda_{r,w}$-modules is related to the Birkhoff problem of classification subgroups of finite abelian groups, see \cite{birkhoff}. Categories of finite dimensional $\Lambda_{r,w}$-modules are wild in general,
see \cite{simson}, that means, for arbitrary $r,w$ there is no hope for a~nice classification of $\Lambda_{r,w}$-modules.
However there ere many results describing properties of
modules over these algebras: e.g. \cite{arnold,bhw,kap,kos-sch,rs,sch,simson}.

The paper is organized as follows.
\begin{itemize}
 \item In Section \ref{sec-prel} we present
 preliminaries
on $\Lambda_{2,w}(k)$-modules
and we formulate the main Theorem \ref{thm-main}.
\item Section \ref{sec-hall-alg} contains definitions and some elementary facts about
Ringel-Hall algebras.
 \item In Section \ref{section_classical} certain
 properties of the Steinitz's classical Hall algebra are collected. Moreover applying properties of the Kostka numbers we prove the existence of some polynomials that are needed in the proof of the main result.
 \item The most computational part of the proof of the main result is  placed in Section \ref{sec-relations}, where some relations in the Ringel-Hall algebra of $\Lambda_{2,w}(k)$ are collected.
\item Finally, we finish in Section \ref{sec-hall-poly} the proof of Theorem \ref{thm-main}.
\end{itemize}

\section{Preliminaries
and the main theorem}\label{sec-prel}

By $\soc N$ we denote the socle of a $k[x]$-module
$N\in \mathcal{N}$, i.e. the sum of all simple submodules. The following fact seems to be well-known.

\begin{lem}\label{lem_summand}
Let $0\neq v\in\soc N$, where $N\in\mathcal{N}$ and let $v=x^{m-1}v'$ for some $v'\in N$, where $m$ is the maximal number such that $x^{m-1}$ divides $v$. Then the $k[x]$-submodule $\langle v'\rangle$ of $N$ generated by $v'$ is a direct summand of $N$ isomorphic to $k[x]/(x^m)$.
\end{lem}

\begin{proof} Follows directly form the description of the finite dimensional nilpotent $k[x]$-modules.
\end{proof}

In what follows we  use the standard notation (see \cite{macd}): Given a partition $\alpha=(\alpha_1,\ldots,\alpha_m)$ we set $|\alpha|=\sum_i\alpha_i$, ${\bf n}(\alpha)=\sum_i(i-1)\alpha_i$ and  we denote by $\alpha'=(\alpha'_1,\ldots,\alpha'_s)$ the dual partition defined by $\alpha_i'=\max\{j:\alpha_j\ge i\}$. We write $(1^r)$ for the partition $(1,\ldots,1)$ ($r$ ones). We write $\mu\subset \lambda$ if $\mu_i\le \lambda_i$ for every $i$.

We denote by $\mathcal{S}=\mathcal{S}(k)$  the category of all triples
$(N',N'',f)$, where $f:N'\to N''$ is a~monomorphism.
The morphisms in $\mathcal{S}$ are defined in the natural way. Note that an~object
$(N_{\alpha},N_{\beta},f)$ gives a~short exact sequence
$$
0\to N_\alpha\mapr{f}{} N_\beta \mapr{}{} N_\gamma \to 0,
$$
where $\gamma=\type(\Coker f)$.

Let $\mathcal{S}_r(k)$ denotes
the full subcategory of $\mathcal{S}(k)$ consisting of all triples isomorphic to $(N_{\alpha}, N_{\beta}, f)$ with $\alpha_1\leq r$.
Note that the category $\mathcal{S}_0$ is equivalent to $\mathcal{N}$.
The categories $\mathcal{S}_r$ for $r\geq 3$ have wild representation type, see \cite{simson}.
It means that it is hopeless to find a~nice parametrization of isomorphism classes of
objects in $\mathcal{S}_r$, for $r\geq 3$.
In the paper we work in the category $\mathcal{S}_2$ that  has a~discrete representation type,
i.e. for any integer $m$ there is only finitely many
(up to isomorphism) objects in $\mathcal{S}_2$ with $k$-dimension $m$.

If the nilpotency degree of the ambient space is bounded by $w$, then an object of $\mathcal{S}_r(k)$ ($r\le w$) can be treated  as a right module over the algebra
$$
\Lambda_{r,w}(k)=\left[\begin{array}{cc}k[x]/(x^r)&k[x]/(x^r)\\0&k[x]/(x^w)\end{array}\right].
$$

\begin{lem}\label{lem_classification}
The algebra $\Lambda_{2,w}(k)$ is representation finite
 (i.e. there is only finitely many isomorphism classes of
 indecomposable $\Lambda_{2,w}(k)$-modules) and every indecomposable $\Lambda_{2,w}(k)$-module is isomorphic to one from the following list.

\begin{enumerate}
\item[(1)]  $P_0^m(k)=(0,N_{(m)}(k),0)$, $1\le m\le w$,
\item[(2)] $P_1^m(k)=(N_{(1)}(k),N_{(m)}(k),\iota)$, $1\le m\le w$,   $\iota(1)=x^{m-1}$,
\item[(3)] $P_2^m(k)=(N_{(2)}(k),N_{(m)}(k),\iota)$, $2\le m\le w$,   $\iota(1)=x^{m-2}$,
\item[(4)] $B_2^{m',m}(k)=(N_{(2)}(k),N_{(m')}(k)\oplus N_{(m)}(k),\iota)$, $3\le m+2\le m'\le w$,  $\iota(1)=\left[\begin{array}{c}x^{m'-2}\\x^{m-1}\end{array}\right]$,
\item[(5)] $P_1^0(k)=(N_{(1)}(k),0,0)$,
\item[(6)] $P_2^0(k)=(N_{(2)}(k),0,0)$,
\item[(7)] $P_2^1(k)=(N_{(2)}(k),N_{(1)}(k),\iota)$, $\iota(1)=1$,
\item[(8)] $Z_2^m(k)=(N_{(2)}(k),N_{(m)},\kappa)$, $\kappa(1)=x^{m-1}$, $1\le m\le w$.
\end{enumerate}
\end{lem}

\begin{proof}
Let $V=(V',V'',f)$ be an indecomposable $\Lambda_{2,w}(k)$-module. Given an element $w$ of $V'$ we denote by $\langle w\rangle$ the $k[x]/(x^2)$-submodule of $V'$ generated by $w$. If $f$ is a monomorphism, then $V$ is isomorphic to a module of one of the types (1)-(4) by \cite{bhw}. Suppose otherwise and let $f(v)=0$ for some $0\neq v\in V'$. If $xv\neq 0$, then the submodule $\langle v\rangle$ of $V'$ is isomorphic to $k[x]/(x^2)$, thus by Lemma \ref{lem_summand} (Let us recall  that $x^2V'=0$.), it is a direct summand of $V'$. It follows that $P_2^0(k)$ is a direct summand of $V$, hence $V\cong P^0_2(k)$. If $xv=0$ then $v\in\soc V'$. If $v$ is not divisible by $x$, then the submodule $\langle v\rangle$ of $V'$ is a direct summand of $V'$ by Lemma \ref{lem_summand} and it is  isomorphic to $k[x]/(x)$. Then  $V$ has a direct summand isomorphic to $P_1^0(k)$, hence $V$ is isomorphic to $P^0_1(k)$.  Otherwise, let $xv'=v$; then the submodule $\langle v'\rangle$ of $V'$ is isomorphic to $k[x]/(x^2)$. Moreover, $f(v')\in\soc V''$, thus, again Lemma \ref{lem_summand}, $f(v')$ is an element of the socle of a direct summand of $V''$ isomorphic to $k[x]/(x^m)$, where $m$ is the maximal number such that $x^{m-1}$ divides $f(v')$. Then $V$ has a direct summand isomorphic to $Z_2^m(k)$, hence $V\cong Z_2^m(k)$.
\end{proof}

The objects $(N_\alpha,N_\beta,f)$ of $\mathcal{S}_2(k)$  with $\beta_1\leq n$  are the $\Lambda_{2,n}(k)$-modules having no direct summands isomorphic to $P_1^0(k)$, $P_2^0(k)$, $P_2^1(k)$ or $Z_2^m(k)$ ($m\leq n$).

Observe that $Z_2^1(k)=P_2^1(k)$.

It follows that the isomorphism classes of indecomposable $\Lambda_{2,n}(k)$-modules are in a 1-1 correspondence with a finite set
\begin{equation}
\begin{split}
 I_{2,n}=\{P_0^m:1\le m\le n\}\cup \{P_1^m:1\le m\le n\}\cup\{P_2^m:2\le m\le n\}\cup  \\ \cup \{B_2^{m',m}:3\le m+2\le m'\le n\}\cup\{P_1^0,P_2^0\}\cup \{Z_2^m:1\leq m\leq n\}
\end{split}
\label{eq-I2n-set}
\end{equation}
 which is independent on the field $k$. We denote by  $X(k)$ the module corresponding to $X\in I_{2,n}$. Thus, for any function $a:I_{2,n}\rightarrow \NN_0$ we can associate the $\Lambda_{2,n}(k)$-module $\bigoplus_{X\in I_{2,n}}X(k)^{a(X)}$, which we shall denote by $a(k)$. It is convenient to write $\bigoplus_{X\in I_{2,n}}X^{a(X)}$ instead of $a$ and call it a $\Lambda_{2,n}$-module.\footnote{More precisely, we can define a $\mathbb{Z}$-algebra $\Lambda_{2,n}$ and $\Lambda_{2,n}$-modules $P_1^m$ etc. in such a~way that $P_1^m(k)=P_1^m\otimes_{\mathbb{Z}}k$. }

 Given right modules $X,Y,Z$ over a finite dimensional algebra $A$ over a finite field $k$ we set
$$
F_{X,Z}^{Y}=|\{U\subseteq Y:U \mbox{ is a submodule of } Y\; ,\;\; U\simeq X \mbox{ and } Y/U\simeq Z\}|.
$$

The main aim of the paper is the following theorem.

\begin{thm}
 \label{thm-main}
 Let $a,b,c:I_{2,n}\rightarrow\NN_0$. There exists $\varphi_{a,c}^b\in \mathbb{Z}[t]$
such that for any finite field $k$
 $$F_{X_a(k),X_c(k)}^{X_b(k)}=\varphi_{a,c}^b(q),$$
 where $q=|k|$.
\end{thm}

The polynomials $\varphi_{a,c}^b$ are called {\bf Hall polynomials}.

\section{Generalities on Ringel-Hall algebras}\label{sec-hall-alg}
Let $k$ be a~finite field,  $A$ be a~finite dimensional $k$-algebra and let $X_1,\ldots,X_t$, $Y$ be (right) $A$-modules. Let $F^Y_{X_1,\ldots,X_t}$ be the number of filtrations
$$
   0=U_0\subseteq U_{1}\subseteq\ldots\subseteq U_{t-1}\subseteq U_t=Y
$$
of $Y$ such that $U_{i}/U_{i-1}\simeq X_i$
for all $i=1,\ldots,t$. Note that $F_{X,Z}^Y$
is the number of submodules $U$ in $Y$ with
$U\simeq X$ and $Y/U\simeq Z$.

Following \cite{ringel} we define the Ringel-Hall algebra $\mathcal{H}(A)$ to be the $\mathbb{C}$-vector space with basis $\{u_X\}_{[X]}$ indexed
by the ismomorphism classes $[X]$ of finite dimensional $A$-modules and with multiplication
given by the following formula
$$
u_Xu_Z=\sum_{[Y]}F_{X,Z}^Yu_Y.
$$

It is well-know that $\mathcal{H}(A)$ is an~associative algebra with $1=u_0$ and
\begin{equation}\label{eq-trzy}
u_{X_1}u_{X_2}u_{X_3}=\sum_{[T]}F_{X_1,X_2,X_3}^Tu_T,
\end{equation}
where $[T]$ runs through the set of the isomorphism classes of finite dimensional $A$-modules, see \cite[Proposition 1]{ringel}.

Note that
\begin{equation}
 \sum_{[T]}F_{X_1,T}^YF_{X_2,X_3}^T=  F_{X_1,X_2,X_3}^Y=\sum_{[T]}F_{X_1,X_2}^TF_{T,X_3}^Y,
\label{eq-assoc-F}
\end{equation}
see \cite[page 442]{ringel}.

More generally,
\begin{equation}\label{eq-wiele}
u_{X_1}\ldots u_{X_t}=\sum_{[T]}F_{X_1,\ldots,X_t}^Tu_T
\end{equation}
and
\begin{equation}\label{eq-assoc-FF}
F_{X_1,\ldots,X_t}^T=\sum_{[Y_1],\ldots,[Y_{t-1}]} F_{X_1,Y_2}^{Y_1}F_{X_2,Y_3}^{X_2}\ldots F_{X_{t-1},X_t}^{Y_{t-1}},
\end{equation}
 where $[Y_1],\ldots,[Y_{t-1}]$ run through the set of the isomorphism classes of finite dimensional $A$-modules, see \cite[Remark on page 4]{peng}.

The following lemma is a~version of \cite[Lemma 2.1]{nasr-isf}.

\begin{lem}\label{lem-reduction}
 Let $X,X_1,X_2,X_3,X_4,Y,Z$ be $A$-modules
 (possibly equal to $0$),
 $a,b\in \mathbb{C}$
 and assume that
 $u_X=au_{X_1}u_{X_2}+bu_{X_3}u_{X_4}$
 in $\mathcal{H}(A)$ Then
   $$F_{X,Z}^Y=aF_{X_1,X_2,Z}^Y+bF_{X_3,X_4,Z}^Y$$
\end{lem}

\begin{proof}
Let $u_X=au_{X_1}u_{X_2}+bu_{X_3}u_{X_4}$. Then by (\ref{eq-trzy})
\begin{multline*}
 (au_{X_1}u_{X_2}+bu_{X_3}u_{X_4})u_Z =
 au_{X_1}u_{X_2}u_Z+bu_{X_3}u_{X_4}u_Z= \\
  (aF_{X_1,X_2,Z}^Y+bF_{X_3,X_4,Z}^Y)u_Y  +\sum_{[Y']\neq [Y]}(aF_{X_1,X_2,Z}^{Y'}+bF_{X_3,X_4,Z}^{Y'})u_{Y'}.
\end{multline*}
Moreover
$$
u_Xu_Z=F_{X,Z}^Yu_Y+\sum_{[Y']\neq [Y]}F_{X,Z}^{Y'}u_{Y'}.
$$
Comparing these formulae we get
$$
F_{X,Z}^Y=aF_{X_1,X_2,Z}^Y+bF_{X_3,X_4,Z}^Y,
$$
because elements $\{u_{Y}\}_{[Y]}$ form a~$\mathbb{C}$-basis of $\mathcal{H}(A)$.
\end{proof}

Let us denote by $Mon_A(X,Y)$ the set of the monomorphisms $X\to Y$.

\begin{lem}\label{lem_injective}
Assume that $X,Y,Z$ are $A$-modules and $X=X'\oplus I$, where $I$ is an injective module. Then $F_{X,Z}^Y$ is nonzero if and only if $Y\cong Y'\oplus I$ for some module $Y'$ and $F_{X',Z}^{Y'}$ is nonzero.

In this case
$$F_{X,Z}^Y=\frac{|Mon_A(I,Y)|}{|Mon_A(I,X)|}\cdot F_{X',Z}^{Y'}.$$
\end{lem}

\begin{proof}  The first assertion follows easily, because $I$ is
 injective. For the proof of the formula note that by (\ref{eq-assoc-F}) applied to $X_1=I$, $X_2=X'$, $X_3=Z$ we get
 $$
 \sum_{[T]}F^{Y}_{I,T}F^T_{X',Z}=\sum_{[U]}F^U_{I,X'}F_{U,Z}^Y
 $$
 and, as $I$ is injective, the equality reduces to
 $$
 F^{Y}_{I,Y'}F^{Y'}_{X',Z}=F^{X}_{I,X'}F_{X,Z}^Y.
 $$
 Since $F^{I\oplus U}_{I,U}=\frac{|Mon_A(I,I\oplus U)|}{|Aut(I)|}$ for injective $I$ and arbitrary $U$, the lemma follows.
\end{proof}

\section{Some relations in the Steinitz's classical Hall algebra}\label{section_classical}

Given a finite field $k$ with $|k|=q$ we denote by $\mathcal{H}(q)$ the (classical) Hall algebra defined as the free $\mathbb{C}$-vector space with the basis $u_{\lambda}$, where $\lambda$ runs through the set of all partitions, equipped with the multiplication defined by the formula
$$
u_{\mu}u_{\nu}=\sum_{\lambda}F_{\mu,\nu}^{\lambda}(q)u_{\lambda},
$$
where $F_{\mu,\nu}^{\lambda}(q)$ is the number of the submodules $U$ of $N(\lambda)$ which are isomorphic to $N(\mu)$ and $N(\lambda)/U\cong N(\nu)$. The classical results mentioned in the Introduction asserts that $F_{\mu,\nu}^{\lambda}(q)$ are polynomials in $q$. In what follows we shall use some more precise properties, which we list in the following proposition.

\begin{prop}\label{prop_wstep}
\begin{enumerate}
\item[(a)] If $F_{\mu,\nu}^{\lambda}(q)\neq 0$ then $\mu,\nu\subseteq \lambda$, $|\lambda|=|\mu|+|\nu|$ and  $F_{\mu,\nu}^{\lambda}$ is a polynomial in $q$ of degree ${\bf n}(\lambda)-{\bf n}(\mu)-{\bf n}(\nu)$.
\item[(b)] $\mathcal{H}(q)$ is generated as a $\mathbb{C}$-vector space by the elements $u_{(1^{l_1})}\ldots u_{(1^{l_r})}$, where $r\in\NN$ and $l_1,\ldots,l_r\in\NN$, and those elements are linearly independent.
\end{enumerate}
\end{prop}
\begin{proof}
This is Chapter II, (4.3) and (2.3) in \cite{macd}.
\end{proof}

For sake of simplicity of further formulations let us define $\mathcal{F}_0$ as the free  associative $\ZZ[t]$-algebra generated by $u_{\lambda}$ as free commutative generators, where $\lambda$ run through all partitions. Given $R\in\mathcal{F}_0$ and a number $q$ we denote by $R(q)$ the element of the $\mathcal{H}(q)$ obtained by evaluating the coefficients of $R$ at $q$.

Let $\CP_n$ denote the set of all partitions of the number $n$. Given a partition $\lambda$ we denote by $v_{\lambda'}$ the product
$$
u_{(1^{\lambda_1'})}\ldots u_{(1^{\lambda_s'})},
$$
where $\lambda'=(\lambda'_1,\ldots,\lambda'_s)$ is the partition dual to $\lambda$.

Recall \cite[I.(6.4)]{macd} that given two partition $\lambda, \mu\in\CP_n$ the Kostka number $K_{\lambda,\mu}$ is defined as the number of tableaux of shape $\lambda$ and weight $\mu$.

\begin{lem}\label{lem_kostka}
The matrix $K=[K_{\lambda,\mu}]_{\lambda, \mu\in\CP_n}$ is invertible and $(K^{-1})_{(n),(1^n)}=(-1)^{n-1}$.
\end{lem}

\begin{proof} Can be found in several places. See eg. \cite[Theorem 3]{duan}, also the proof of Lemma 5 there or \cite[I.6]{macd}, Example 4(d), formula (2).
\end{proof}

The aim of this section is the following proposition.
\begin{prop}\label{prop_Justyna} For every $n\in\NN$ there exists $h_n\in\ZZ[t]$ with the leading coefficient $\pm 1$ and $R_n\in\mathcal{F}_0$ which is a linear combination of elements of the form $u_{(1^{l_1})}\ldots  u_{(1^{l_r})}$, where $r>1$ (consequently $l_i<n$ for every $i=1,\ldots,r$) such that for every finite field $k$
$$
\sum_{\mu\in\CP_n}u_{\mu}=h_n(q)u_{(1^n)}+R_n(q)
$$
in $\mathcal{H}(q)$.
\end{prop}

\begin{proof} Most of the proof is a compilation of well-known facts on the classical Hall algebras.
Let $\lambda \in\CP_n$. Then (see \cite[Chapter II, (2.3)]{macd})
$$
v_{\lambda'}=\sum_{\mu\in\CP_n}a_{\lambda,\mu}u_{\mu},
$$
where $a_{\lambda,\mu}=0$ unless  ${\bf n}(\mu)\ge {\bf n}(\lambda)$ and $a_{\lambda,\lambda}=1$.
Moreover, it follows (see eg. \cite[(2.5)]{Schiff}) that
$$
a_{\lambda,\mu}=\sum_{(\nu_1,\ldots,\nu_{s-1})}F_{(1^{\lambda_1'}),\nu_{s-1}}^{\nu_s}F_{(1^{\lambda_2'}),\nu_{s-2}}^{\nu_{s-1}}\ldots
F_{(1^{\lambda_{s-1}'}),\nu_{1}}^{\nu_2},
$$
where $\nu_1=(1^{\lambda'_s})$, $\nu_s=\mu$ and $(\nu_2,\ldots,\nu_{s-1})$ runs through all sequences of partitions. Clearly, the r.h.s sum is finite. Moreover, by \cite[Example 2.4]{Schiff} and Proposition \ref{prop_wstep} every nonzero summand of this sum is a monic polynomial in $q$ of degree
$$
\sum_{ j=2}^s({\bf n}(\nu_j)-{\bf n}(1^{\lambda_{s-j+1}'})-{\bf n}(\nu_{j-1}))={\bf n}(\nu_s)-{\bf n}{ (\nu_1)}-\sum_{ j=2}^s{\bf n}(1^{\lambda_{s-j+1}'})={\bf n}(\mu)-{\bf n}(\lambda)
$$
and the number of nonzero summands is equal to the number of the tableaux of shape $\mu'$ and weight $\lambda'$ (see \cite[I.1]{macd}), thus it is the Kostka number $K_{\mu',\lambda'}$ (\cite[I.(6.4)]{macd})\footnote{It follows by  \cite[(AZ.4) page 200]{macd}.}.
Let $A$ be the matrix $[a_{\lambda,\mu}]_{\lambda,\mu\in\CP_n}$.

We summarize our observations in the following Claim 1:
\medskip

 $A$ is upper triangular
({ with respect to a~linear ordering of the partitions $\lambda\ge_l\mu$
implying ${\bf n}(\lambda)\ge {\bf n}(\mu)$}), with 1 on the diagonal and the element $a_{\lambda,\mu}$ is a polynomial in $q$ of degree ${\bf n}(\mu)-{\bf n}(\lambda)$ and the leading term $K_{\mu',\lambda'}$.

The matrix $A$ is invertible. Let $A^{-1}=[a'_{\lambda,\mu}]_{\lambda,\mu\in\CP_n}$. The elements $a'_{\lambda,\mu}$ are polynomials in $q$.
\medskip

Claim 2. $\deg(a'_{\lambda,\mu})\le {\bf n}(\mu)-{\bf n}(\lambda)$. In order to prove it, note that $a'_{\lambda,\mu}$ is a linear combination (with integral coefficients) of elements of the form
$$
{\bf a}_{\sigma}=\prod_{\nu}a_{\nu,\sigma(\nu)},
$$
where $\nu$ runs through the partitions in $\CP_n$ different than $\mu$ and $\sigma:\CP_n\setminus\{\mu\}\rightarrow \CP_n\setminus\{\lambda\}$ is a bijection. Thus, if ${\bf a}_{\sigma}\neq 0$,  $$\deg({\bf a}_{\sigma})= \sum_{\nu}\deg(a_{\nu,\sigma(\nu)})={\bf n}(\mu)-{\bf n}(\lambda),$$  the last equality by Claim 1. Thus the Claim 2. follows.
\medskip

Claim 3. $\deg(a'_{(n),(1^n)})={\bf n}(1^n)-{\bf n}((n))=\frac{n(n-1)}{2}>{\bf n}(\mu)-{\bf n}(\lambda)$ whenever $\lambda\neq (n)$ and the leading coefficient of $a'_{(n),(1^n)}$ equals $(-1)^{n-1}$. Indeed, if $\lambda=(n)$, $\mu=(1^n)$, then by Claim 2 every nonzero ${\bf a}_{\sigma}$  has degree ${\bf n}(1^n)-{\bf n}((n))$ and the leading coefficient of $a'_{(n),(1^n)}$ is the element at the position $((n),(1^n))$ of the matrix inverse to $[K_{\mu',\lambda'}]_{\lambda,\mu\in\CP_n}$. Observe that this element is equal to the element at the position $((n),(1^n))$ of the matrix inverse to $[K_{\lambda,\mu}]_{\lambda,\mu\in\CP_n}$,
 hence by Lemma \ref{lem_kostka}  equals to $(-1)^{n-1}$. The claim follows.
\medskip

Now,
\begin{equation}\label{eq_ulambda}
u_{\mu}=\sum_{\lambda}a'_{\mu,\lambda}v_{\lambda'}=a'_{\mu,(1^n)}v_{(1^n)'}+ \sum_{\lambda:\lambda\neq (1^n)}a'_{\mu,\lambda}v_{\lambda'}
\end{equation}
for every $\mu$. Observe that $v_{(1^n)'}=u_{(1^n)}$ and if $\lambda\neq (1^n)$, then $v_{\lambda'}$ is a product of at least two elements of the form $u_{(1^l)}$. Summing up the above equalities over $\mu\in\CP_n$ we obtain what we want, since by the Claims 2. and 3.
all polynomials $a'_{\mu,(1^n)}$, with $\lambda\neq (n)$ have degree strictly less than ${\bf n}(1^n)-{\bf n}((n))$, whereas the degree of $a'_{(n),(1^n)}$ is equal to ${\bf n}(1^n)-{\bf n}((n))$ and its leading coefficient equals $\pm 1$.
\end{proof}

\begin{ex}
The matrix $[a_{\lambda,\mu}]$ for $n=4$ is the following (the rows and columns are indexed by the partitions of $4$ ordered as follows:
$(4)$, $(3,1)$, $(2,2)$, $(2,1,1)$, $(1,1,1,1)$.)
 \small $$
\left[\begin{array}{ccccc}
1&3q+1&2q^2+3q+1&3q^3+5q^2+3q+1&q^6+3q^5+5q^4+6q^3+5q^2+3q+1\\
0&1&q+1&2q^2+q+1&q^5+2q^4+3q^3+3q^2+2q+1\\
0&0&1&q+1&q^4+q^3+2q^2+q+1\\
0&0&0&1&q^3+q^2+q+1\\
0&0&0&0&1
\end{array}\right]
$$
and its inverse is the following:

{\small $$
\left[\begin{array}{ccccc}
1&-3q-1&q^2+q&2q^3-2q^2&-q^6+2q^5+3q^4+3q^3+4q^2+q\\
0&1&-q-1&-q^2+q&q^5-q\\
0&0&1&-q-1&q^3+q\\
0&0&0&1&-q^3-q^2-q-1\\
0&0&0&0&1
\end{array}\right].
$$}

It follows that
$$
\begin{array}{l}
 \sum_{\mu\in\CP_n}u_{\mu}=(-q^6+3q^5+3q^4+3q^3+3q^2)u_{(1^4)}+u_{(1)}^4-3qu_{(1^2)}u_{(1)}^2+\\
 +q^2u^2_{(1^2)}+(2q^3-3q^2-1)u_{(1^3)}u_{(1)}.
 \end{array}
$$
\end{ex}

\begin{lem}\label{lem_relacja2}
$$
\frac{q^n-1}{q-1}u_{(1^n)}=-u_{(2,1^{n-2})}+u_{(1^{n-1})}u_{1}.
$$
\end{lem}

 \begin{proof}
 The middle term of an extension of $k[x]/(x)$ by $(k[x]/(x))^{n-1}$ is isomorphic either to $(k[x]/(x))^{n}$ or to $k[x]/(x^2)\oplus (k[x]/(x))^{n-2}$. The number of submodules of $(k[x]/(x))^{n}$ isomorphic to $(k[x]/(x))^{n-1}$ equals $\frac{q^n-1}{q-1}$ and there is unique submodule of $k[x]/(x^2)\oplus (k[x]/(x))^{n-2}$ isomorphic to $(k[x]/(x))^{n-1}$. The factor module is isomorphic to $(k[x]/(x))^{n}$ in every case. It follows that
 $$
 u_{(1^{n-1})}u_1=\frac{q^n-1}{q-1}u_{(1^{n-1})}+u_{(2,1^{n-2})}.
 $$
 \end{proof}

\section{Some relations in  Hall algebras of $\Lambda_{2,w}(k)$}\label{sec-relations}

In this section we fix a~finite field $k$ and for any
$X\in I_{2,w}$ we write $X=X(k)$.

\begin{rem}
The category $\mathcal{S}$ (treated as a subcategory of the category of $\Lambda_{2,w}(k)$-modules) is closed under extensions.
\end{rem}

\begin{lem}\label{lem-reduction-forms}
Let $|k|=q$ and $w\ge 2$. The following equalities hold in $\mathcal{H}(\Lambda_{2,w}(k))$:
\begin{enumerate}
 \item $u_{P_1^n}=u_{P_0^n}u_{P_1^0}-u_{P_1^0}u_{P_0^n}$ for $w\ge n\geq 1$;
 \item $u_{P_2^1}=u_{P_0^1}u_{P_2^0}-u_{P_2^0}u_{P_0^1}$;
 \item $u_{Z_2^n}=u_{Z_2^{n-1}}u_{P_0^1}-u_{P_0^1}u_{Z_2^{n-1}}+u_{P_2^n}$ for $w\geq n>2$ and $u_{Z_2^2}=u_{Z_2^{1}}u_{P_0^1}-\frac{1}{q}u_{P_0^1}u_{Z_2^1}+\frac{1}{q}u_{P_2^2}$.
 \item $u_{P_2^n}=u_{P_0^n}u_{P_2^0}-u_{P_2^0}u_{P_0^n}-u_{Z_2^n}$ for $w\ge n\geq 2$;
 \item $u_{P_2^n\oplus P_0^1}=\frac{1}{q}(u_{P_2^n}u_{P_0^1}-u_{P_2^{n+1}})$ for $w>n\geq 3$ and $u_{P_2^w\oplus P_0^1}=\frac{1}{q}u_{P_2^w}u_{P_0^1}$;
 \item $u_{B_2^{n,1}}=\frac{1}{q}(u_{P_1^n}u_{P_1^1}-u_{P_1^1}u_{P_1^n}+u_{P_2^{n+1}})$ for $w>n\geq 3$ and
 $u_{B_2^{w,1}}=\frac{1}{q}(u_{P_1^w}u_{P_1^1}-u_{P_1^1}u_{P_1^n})$;
 \item $u_{P_2^n\oplus P_0^m}=\frac{1}{q^m}u_{P_2^n}u_{P_0^m}-\sum_{s=1}^{\max(w-n,m-1)}\frac{q-1}{q^{s+1}}u_{P_2^{n+s}\oplus P_0^{m-s}}-  \frac{1}{q^m}u_{P_2^{n+m}}$ 
 and
  $q^mu_{B_2^{n,m}}=u_{P_1^n}u_{P_1^m}-u_{P_1^m}u_{P_1^n}+q^{m-2}u_{B_2^{n+1,m-1}}+(q-1)q^{m-1}u_{P_2^{n+1}\oplus P_0^{m-1}}$
 for $n-2\geq m\geq 2$.  We agree that $P_0^0=0$.
\end{enumerate}
\end{lem}

\begin{proof}
 Let $w\geq 2$ and $A=\Lambda_{2,w}(k)$.

 \begin{enumerate}
  \item Let $n\geq 1$. It is easy to see that $F_{P_0^n,P_1^0}^{P_1^n}=1$,

 $$\Hom_A(P_0^n,P_1^0)=\Hom_A(P_1^0,P_0^n)=\Ext^1_A(P_0^n,P_1^0)=0,$$
 and the only extensions of $P_1^0$ by $P_0^n$ are: $P_1^n$, $P_0^n\oplus P_1^0$.
 Therefore, $u_{P_0^n}u_{P_1^0}=u_{P_1^n}+u_{P_0^n\oplus P_1^0}$ and
 $u_{P_1^0}u_{P_0^n}=u_{P_0^n\oplus P_1^0}$.
 We are done.

 \item The proof is similar to the proof of the statement 1.

  \item Let $n\geq 2$. Note that there is no monomorphism $Z_2^{n-1}\to P_2^n$ neither $Z_2^{n-1}\to B_2^{n-1,1}$. Therefore the only extensions of $P_0^1$ by $Z_2^{n-1}$ are: $Z_2^n$ and $P_0^1\oplus Z_2^{n-1}$. Therefore
  $u_{Z_2^{n-1}}u_{P_0^1}=u_{Z_2^n}+qu_{P_0^1\oplus Z_2^{n-1}}$ for $n>2$ and $u_{Z_2^{1}}u_{P_0^1}=u_{Z_2^2}+u_{P_0^1\oplus Z_2^{1}}$. On the other hand, $F_{P_0^1,Z_2^{n-1}}^{P_0^1\oplus Z_2^{n-1}}=q$,
  $F_{P_0^1,Z_2^{n-1}}^{P_2^n}=1$
 and  there are no epimorphisms: $Z_2^n\to Z_2^{n-1}$,  $B_2^{n-1,1}\to Z_2^{n-1}$. It follows that $u_{P_0^1}u_{Z_2^{n-1}}=qu_{P_0^1\oplus Z_2^{n-1}}+u_{P_2^n}$ for $n\ge 2$. We are done.
%

  \item Let $n\geq 2$. It is easy to see that $F_{P_0^n,P_2^0}^{P_2^n}=F_{P_0^n,P_2^0}^{Z_2^n}=1$,

 $$\Hom_A(P_0^n,P_2^0)=\Hom_A(P_2^0,P_0^n)=\Ext^1_A(P_0^n,P_2^0)=0,$$
 and the only extensions of $P_2^0$ by $P_0^n$ are: $P_2^n$, $Z_2^n$ and $P_0^n\oplus P_2^0$.
 Therefore, $u_{P_0^n}u_{P_2^0}=u_{P_2^n}+u_{Z_2^n}+u_{P_0^n\oplus P_2^0}$ and
 $u_{P_2^0}u_{P_0^n}=u_{P_0^n\oplus P_2^0}$.
 We are done.

 \item Let $n\geq 3$. The only extensions of $P_0^1$ by $P_2^n$ are: $P_2^{n+1}$ (if $n<w$) and
 $P_0^1\oplus P_2^n$, because $P_2^n,P_0^1\in \mathcal{S}_2$,
 the category $\mathcal{S}$ is closed under extensions and there is no monomorphism $P_2^n\to B_2^{n,1}$.
 It is straightforward to check that $F_{P_2^n,P_0^1}^{P_2^{n+1}}=1$ (if $n<w$) and $F_{P_2^n,P_0^1}^{P_2^n\oplus P_0^1}=q$. We are done.

 \item Let $n\ge 3$. It is straightforward to check that $F_{P_1^n,P_1^1}^{B_2^{n,1}}=F_{P_1^n,P_1^1}^{P_1^1\oplus P_1^n}=F_{P_1^1,P_1^n}^{P_1^1\oplus P_1^n}=q$ and
 $F_{P_1^1,P_1^n}^{P_2^{n+1}}=1$. The only extensions of $P_1^1$ by $P_1^n$ are $B_2^{n,1}$ and $P_1^1\oplus P_1^n$, because the category $\mathcal{S}$
 is closed under extensions and there is no epimorphism
 $P_2^{n+1}\to P_1^1$.
 Moreover, the only extensions of $P_1^n$ by $P_1^1$ are $P_2^{n+1}$ and $P_1^1\oplus P_1^n$, because the category $\mathcal{S}$
 is closed under extensions and there is no epimorphism
 $B_2^{n,1}\to P_1^n$.
 Therefore, $u_{P_1^n}u_{P_1^1}=qu_{B_2^{n,1}}+qu_{P_1^1\oplus P_1^n}$ and $u_{P_1^1}u_{P_1^n}=u_{P_2^{n+1}}+qu_{P_1^1\oplus P_1^n}$. We are done.

 \item Let $n-2\geq m\geq 2$ and $m>s\geq 1$. It is easy to check that $F_{P_2^n,P_0^m}^{P_2^{n+m}}=1$, $F_{P_2^n,P_0^m}^{P_2^n\oplus P_0^m}=q^m$, $F_{P_2^n,P_0^m}^{P_2^{n+s}\oplus P_0^{m-s}}=(q-1)q^{m-s-1}$ and the only extensions of $P_0^m$ by $P_2^n$ are  $P_2^{n+m}$ and $P_2^{n+s}\oplus P_0^{m-s}$ for all $m>s\geq 1$. It is because the category $\mathcal{S}$ is closed under extensions and there are no monomorphisms
 $P_2^n\to B_2^{n+s,m-s}$, $P_2^n\to P_0^{n+s}\oplus P_2^{m-s}$ for $s\geq 0$. The first formula follows.

 We divide the proof of the second formula into several
 steps.
 \begin{enumerate}
  \item For $s=1,\ldots,m-1$ we have
  $F_{P_1^m,P_1^n}^{P_1^{n+s}\oplus P_1^{m-s}}=F_{P_1^n,P_1^m}^{P_1^{n+s}\oplus P_1^{m-s}}.$ This follows since the  Steinitz's classical Hall algebra is commutative and
   there is 1-1 correspondence between exact sequences
  $0\to P_1^n \to P_1^{n+s}\oplus P_1^{m-s}\to P_1^m\to 0$ (resp. $0\to P_1^m \to P_1^{n+s}\oplus P_1^{m-s}\to P_1^n\to 0$) and the exact sequences of the ambient spaces: $0\to N_{(n)} \to N_{(n+s)}\oplus N_{(m-s)}\to N_{(m)}\to 0$ (resp. $0\to N_{(m)} \to N_{(n+s)}\oplus N_{(m-s)}\to N_{(n)}\to 0$).

  \item For $s=2,\ldots,m-1$ we have
  $F_{P_1^m,P_1^n}^{B_2^{n+s,m-s}}=F_{P_1^n,P_1^m}^{B_2^{n+s,m-s}} =q^{m-s-1}(q-1)$.
  \item We have
  $F_{P_1^m,P_1^n}^{B_2^{n+1,m-1}}=q^{m-2}(q-2)$   and $F_{P_1^n,P_1^m}^{B_2^{n+1,m-1}}=q^{m-2}(q-1).$
  \item $F_{P_1^n,P_1^m}^{B_2^{n,m}}=q^m$ and
  $F_{P_1^m,P_1^n}^{B_2^{n,m}}=0$, because there is no epimorphism $B_2^{n,m}\to P_1^n$.
  \item $F_{P_1^n,P_1^m}^{P_2^{n+s}\oplus P_0^{m-s}}=0$ for $s\geq 0$, because there is no epimorphisms
  $P_2^{n+s}\oplus P_0^{m-s}\to P_1^m$.
  \item $F_{P_1^m,P_1^n}^{P_2^{n+s}\oplus P_0^{m-s}}=0$ for $s=0$ and $s>1$, because there is no epimorphisms
  $P_2^{n+s}\oplus P_0^{m-s}\to P_1^n$.
  \item $F_{P_1^m,P_1^n}^{P_2^{n+1}\oplus P_0^{m-1}}=q^{m-1}(q-1)$.
  \item straightforward calculations show that$F_{P_1^m,P_1^n}^{P_0^{n+s}\oplus P_2^{m-s}}=0=F_{P_1^n,P_1^m}^{P_0^{n+s}\oplus P_2^{m-s}}$ for $s=1,\ldots m-1$.
 \end{enumerate}
The fact that the category $\mathcal{S}$ is closed under
etensions and the statements (a)-(h) imply the second formula.
  \end{enumerate}
  This finishes the proof.
\end{proof}

%

\section{The proof of the main result}
\label{sec-hall-poly}

As in Section \ref{section_classical} let us define $\mathcal{F}_0$ as the $\ZZ[t]$-algebra generated by $u_a$ as free  associative algebra, where $a\in\NN_0^{I_{2,w}}$. Given $R\in\mathcal{F}$ and a number $q$ we denote by $R(q)$ the element of the $\mathcal{H}(\Lambda_{2,w}(k))$, where $q=|k|$,  obtained by evaluating the coefficients of $R$ at $q$.

The set of the isomorphism classes of finite-dimensional $\Lambda_{2,w}(k)$-modules is in a natural 1-1 correspondence with the set $\NN_0^{I_{2,w}}$ of functions $a:I_{2,w}\rightarrow\NN_0$. We say that $a\in \NN_0^{I_{2,w}}$ {\em is in the polynomial zone} if for every $b,c\in \NN_0^{I_{2,w}}$ there exists $\varphi_{a,c}^b\in \mathbb{Z}[t]$
such that for any finite field $k$,
 $F_{a(k),c(k)}^{b(k)}=\varphi_{a,c}^b(q),$
 where $q=|k|$. Let us denote by $\PP_w$ the polynomial zone, that is, the set of the $a\in I_{2,w}$ which are in the polynomial zone. We need to prove that $\PP_w=\NN_0^{I_{2,w}}$. In what follows we identify functions $a$ with the corresponding modules.

 If $w\le w'$ then $I_{2,w}\subset I_{2,w'}$ and we treat $\NN_0^{I_{2,w}}$ as a subset of $\NN_0^{I_{2,w'}}$ in an obvious way: we extend a function defined on $I_{2,w}$ to $I_{2,w'}$ by putting zeros outside $I_{2,w}$.
 Given a partition $\alpha=(\alpha_1,\ldots,\alpha_l)$ we denote by $P^{\alpha}_0$ the module $(0,N(\alpha),0)=\bigoplus_{i=1}^lP^{\alpha_i}_0$.

 \begin{lem}\label{lem_embed}
 If $w\le w'$  then  {$\NN_0^{I_{2,w}}\cap \PP_{w'}\subseteq\PP_w$}.
 \end{lem}

 \begin{proof}
 Obvious.
 \end{proof}

 \begin{lem}\label{lem_simple}  For every $w\in \NN$ the simple modules, that is $P^1_0$ and $P^0_1$ are in the polynomial zone $\PP_w$.
 \end{lem}

 \begin{proof}
  Let $b,c\in \mathbb{N}_0^{I_2,w}$ be such that
 $\dim_kb(k)-\dim_kc(k)=1$.
 Note that kernels of all epimorphisms
 $b(k)\to c(k)$ are isomorphic to $P_0^1$ or kernels of all epimorphisms
 $b(k)\to c(k)$ are isomorphic to $P_1^0$. It follows that for $a\in\{ P_0^1,P_1^0\}$ the number
 $F_{a(k),c(k)}^{b(k)}$ equals the number of all subobjects $U$ of $b(k)$
 such that $b(k)/U\simeq c(k)$ and it is given by a~polynomial by Ringel's results in \cite[(4) at page 441]{ringel}.
 \end{proof}

 \begin{lem}\label{lem_inductive}
Let $R\in\mathcal{F}_0$ be a linear combination, with coefficients in $\ZZ[t]$, of products of elements $u_b$ with $b\in\PP_w$.
Assume that $a\in\NN_0^{I_{2,w}}$ and there exists $h\in\ZZ[t]$ with the leading coefficient equals to $\pm 1$ and  such that
\begin{equation}\label{eq_ind}
h(q)\cdot u_a=R(q)
\end{equation}
in the Hall algebra $\mathcal{H}(\Lambda_{2,w}(k))$ for every finite field $k$, where $q=|k|$.
Then $a\in\PP_w$.
\end{lem}
\begin{proof} Let $b,c\in I_{2,w}$. By multiplying the equation (\ref{eq_ind}) by $u_c$ we obtain
$$
h(q)\cdot u_au_c=R(q)u_c.
$$
Comparing the coefficients at the basic element $u_b$ we get
$$
h(q)F_{a(k),c(k)}^{b(k)}(q)=R_{b,c}(q)
$$
where $R_{b,c}(q)$  depends polynomially on $q=|k|$ by the formula (\ref{eq-assoc-FF})  thanks to the fact that $R$ is a linear combination of products of elements $u_b$ with $b\in\PP_w$.  Therefore, by \cite[Lemma, p. 441]{ringel} also $F_{a(k),c(k)}^{b(k)}(q)$ depends polynomially on $q$, because the leading coefficient of $h$ is equal to $\pm 1$. We have proved that $a$ is in the polynomial zone.
\end{proof}

\begin{lem}\label{lem_semisimple} For every $m,w\in\NN$, the module $(P^1_0)^m$ is in the polynomial zone $\PP_w$.
\end{lem}

\begin{proof} We proceed by induction on $m$. For $m=1$ the assertion follows by Lemma \ref{lem_simple}. Assume that $m>1$ and  $(P^1_0)^s$ is in polynomial zone for every $1\le s< m$. By Lemma \ref{lem_embed} we can assume that $w\ge m$. Then, by Proposition \ref{prop_Justyna}
\begin{equation}\label{form_semisimple1}
\sum_{\mu\in\CP_m}u_{\mu}=h_m(q)u_{(P_0^1)^m}+R_m(q)
\end{equation}
in $\mathcal{H}(\Lambda_{2,w}(k))$, where $h_m(q)$ is a nonzero  polynomial in $q$ with the leading coefficient $\pm 1$ and $R_m\in\mathcal{F}$ is a linear combination of elements of the form $u_{(P_0^1)^{l_1}}\ldots  u_{(P_0^1)^{l_r}}$ with $r\ge 2$. 

Let $b,c\in \mathbb{N}^{I_{2,w}}$. By multiplying
\ref{form_semisimple1} by $u_c$ in the Hall algebra and
 comparing the coefficients at the basic element $u_b$ we get
\begin{equation}\label{form_semisimple2}
\sum_{\lambda\in\CP_m}F_{P^{\lambda}_0(k),c(k)}^{b(k)}(q)=h_m(q)F_{(P^1_0(k))^m,c(k)}^{b(k)}(q)+\widetilde{R}_{m,b,c},
\end{equation}
where $\widetilde{R}_{m,b,c}$ is a linear combination (with scalar coefficients) of products of numbers of the form $F_{(P^1_0(k))^l,c'(k)}^{b'(k)}(q)$ with $l<m$ and $c',b'\in\NN_0^{I_{2,w}}$, to see this combine formulae (\ref{eq-assoc-FF}) and (\ref{form_semisimple2}). By the induction hypothesis, $\widetilde{R}_{m,b,c}$ depends polynomially on $q=|k|$. On the other hand, the l.h.s of (\ref{form_semisimple2}) is the number of all subobjects $U$ of $b(k)$ such that $b(k)/U\simeq c(k)$. It is known by \cite{ringel} that this number depends polynomially on $q$.  Therefore, by \cite[Lemma, p. 441]{ringel} so does $F_{(P^1_0(k))^m,c(k)}^{b(k)}(q)$, because the leading coefficient of $h_m$ is equal to $\pm 1$. We have proved that $(P^1_0(k))^m$ is in the polynomial zone.
\end{proof}

\begin{lem}\label{lem_semisimple2} For every $m,w\in\NN$, the module $(P^0_1)^m$ is in the polynomial zone $\PP_w$.
\end{lem}
\begin{proof} We proceed by induction on $m$. For $m=1$ the assertion follows by Lemma \ref{lem_simple}. Assume that $m>1$ and  $(P^0_1)^s$ is in polynomial zone for every $1\le s < m$.
The $\Lambda_{2,w}(k)$-module $P^0_2(k)$ is injective, thus by Lemma \ref{lem_injective} $P^0_2\oplus (P^0_1)^{m-2}$ is in the polynomial zone\footnote{We apply the fact that the number $|Mon(a(k),b(k))|$ depends polynomially on $q=|k|$, see \cite[p. 441]{ringel}.}, as $(P^0_1)^{m-2}$ is in the polynomial zone by the induction hypothesis.
Now $(P^0_1)^m$ is in the polynomial zone thanks to Lemma \ref{lem_inductive} and Lemma \ref{lem_relacja2}.

\end{proof}

\begin{lem}\label{lem_P0n} The  modules  $P^n_0$ are  in the polynomial zone for every $1\le n\le w$.
\end{lem}

\begin{proof} For $\lambda=(n)$ the formula (\ref{eq_ulambda}) yields
\begin{equation}
u_{P^{(n)}_0}=\sum_{\lambda}a'_{(n),\lambda}v_{\lambda'}=a'_{(n),(1^n)}u_{((P^1_0)^n)}+ \sum_{\lambda:\lambda\neq (1^n)}a'_{(n),\lambda}u_{((P^1_0)^{\lambda'_1})}\ldots u_{((P^1_0)^{\lambda'_s})},
\end{equation}
where we set $\lambda'=(\lambda'_1,\ldots,\lambda'_s)$. It follows by Lemma \ref{lem_semisimple} and Lemma \ref{lem_inductive} that $P^{(n)}_0$ is in the polynomial zone. We use the fact that the leading coefficient of the polynomial $a'_{(n),(1^n)}$ is $\pm 1$ (Claim 3 in the proof of Proposition \ref{prop_Justyna}).
\end{proof}

\begin{lem} \label{lem_indecomposable} Every indecomposable module is in polynomial zone.
\end{lem}

\begin{proof} Let $N$ be one of the indecomposable modules listed in Lemma \ref{lem_classification}. We are going to prove that $N$ is in the polynomial zone. If $N$ is of the form  $P^n_0$ then this follows by Lemma \ref{lem_P0n}. This is also true for $N=P_1^0$ by Lemma \ref{lem_simple} and for $N=P^0_2$ by the injectivity and Lemma \ref{lem_injective}.

Now by applying Lemma \ref{lem_inductive}  and Lemma \ref{lem-reduction-forms} (1)-(3) we get the assertion for $N$ of the form $P^n_1$ for $n\ge 1$ and $P^1_2=Z_2^1$ and $Z_2^n$ for $n\ge 2$.

Having this, we prove similarly  by Lemma \ref{lem-reduction-forms} (4) that $P_2^n$ are in the polynomial zone for $2\le n\le w$.

In the next step,  by Lemma \ref{lem-reduction-forms} (6) we prove the assertion for $N=B_2^{n,1}$ for $n\ge 3$, again applying Lemma \ref{lem_inductive}. At the same time we show by Lemma \ref{lem-reduction-forms} (5) the the decomposable modules
$P_2^n\oplus P_0^1$, $n\ge 3$  are in the polynomial zone.

We shall prove that for every $n,m\le w$ the modules  $P_2^n\oplus P_0^m$ and $B^{n,m}_2$ are in the polynomial zone by the induction on $m$. For $m=1$ the assertion is proved above. Assume that $m>1$ and the modules $P_2^n\oplus P_0^{m-s}$ and $B^{n,m-s}_2$ are in the polynomial zone for every $m>s>1$. Then our assertion follows by Lemma \ref{lem-reduction-forms} (7) and Lemma \ref{lem_inductive}.
\end{proof}

Now Theorem \ref{thm-main} follows by \cite[Theorem 2.9]{nasr-isf}:

 \begin{thm}
  Let $A$ be a representation finite algebra over finite field $k$. Then $A$
has a~Hall polynomials if and only if for each $X, Z \in \mod A$ and $Y \in \ind A$, Hall
polynomial $\varphi^X_{Y,Z}$ exist.
 \end{thm}

\pagebreak

Address of the authors:

\parbox[t]{5.5cm}{\footnotesize\begin{center}
              Faculty of Mathematics\\
              and Computer Science\\
              Nicolaus Copernicus University\\
              ul.\ Chopina 12/18\\
              87-100 Toru\'n, Poland\end{center}}
\parbox[t]{5.5cm}{\footnotesize\begin{center}
              Faculty of Mathematics\\
              and Computer Science\\
              Nicolaus Copernicus University\\
              ul.\ Chopina 12/18\\
              87-100 Toru\'n, Poland\end{center}}

\smallskip \parbox[t]{5.5cm}{\centerline{\footnotesize\tt skasjan@mat.umk.pl}}
           \parbox[t]{5.5cm}{\centerline{\footnotesize\tt justus@mat.umk.pl}}

\end{document}